\documentclass[a4paper,twoside,reqno,14pt]{amsart}
\RequirePackage{amsmath,amsthm,amssymb,amsfonts}
\RequirePackage{hyperref}% for electronic edition 

\newcommand\thistitle{Modular Metric Spaces}%full title.
\newcommand\thistitleshort{Modular Metric Spaces}% short version for running head.
\title[\thistitleshort]{\thistitle}%do not remove this

\newcommand\thisauthor{Hana Abobaker and Raymond A. Ryan}%full name(s)
\newcommand\thisauthorshort{Abobaker and  Ryan}%surname(s) for running head
\author[\thisauthorshort]{\thisauthor}

\address[Hana Abobaker]{Omar Al-Mukhtar University, Al Bayda, Libya}  %No \\ in address.
\address[Raymond A. Ryan]{School of Mathematics, Statistics \&
Applied Mathematics, NUI Galway, Ireland.}
\email[Hana Abobaker]{libya.first1990@gmail.com}
\email[Raymond A. Ryan]{ray.ryan@nuigalway.ie}
\thanks{The work of the first author was supported by a grant from the Ministry of Education of Libya, to whom she expresses her gratitude.}

% helpful keywords
% AMS 2010 subject classifications.

\newtheorem*{Theorem*}{Theorem}
\newtheorem{example}{Example}[section]
\newtheorem{theorem}{Theorem}[section]

\theoremstyle{definition}
\newtheorem{defi}{Definition}

\theoremstyle{definition}
\newtheorem{proposition}[theorem]{Proposition}

\newcommand\authorbio{
{\bf Hana Abobaker} was a student in the MSc in
Mathematics at National University of Ireland Galway
in 2015--16, supported by a grant from the Mininstry
of Education of Libya. She is currently working in
Omar Al-Mukhtar University, Al Bayda, Libya.
\\
{\bf Raymond A. Ryan} is a long standing member of
the School of Mathematics, Statistics \& Applied Mathematics at NUI Galway. He is the author of the
book \textit{Introduction to Tensor Products of Banach Spaces}, published by Springer.
}
\begin{document}
\maketitle

%%%%%%%%%% AUTHOR: ADD YOUR ABSTRACT:
\begin{abstract}
We give a short introduction to the  theory of modular metric spaces. 
This is a corrected version of the paper \cite{AR},
which had some errors.  We are grateful to V. V. Chistyakov
for bringing these to our attention.
\end{abstract}

%%%%%%%%%%%%%%%%%%%%%%%  AUTHOR's MAIN TEXT BEGINS HERE:
%%%%%%%%%%%%%%%%%%%%%%%  Use amsart conventions. 
%%%%%%%%%%%%%%%%%%%%%%

\section{Introduction}
This is a corrected version of the paper \cite{AR},
which had some errors.  We are grateful to V. V. Chistyakov
for bringing these to our attention.

We begin with a simple motivating example.  
Let $X$ be the set of all points above water on the earth's surface.  
For two points $x$, $y$, let us denote by $v_t(x,y)$ the average speed required to travel over land 
from $x$ to $y$ in a time $t$.  What are the properties of the
function $v_t(x,y)$?

Clearly, if we fix $x$ and $y$, then $v_t(x,y)$ takes nonnegative values and is a non-increasing function of $t$. And 
of course, this function is symmetric in $x$ and $y$.  But there is an issue we have glossed over --- what if $x$ and $y$ lie in different landmasses?  Since we are required to 
travel by land, it is impossible to get from $x$ to $y$ in 
a time $t$, no matter how fast we travel.  As we would like our speed function to be defined in all circumstances, it is reasonable to allow extended real values and to assign the value $v_t(x,y)=\infty$ in this case.  

To summarize, we now have a function taking $t>0$ and $x$, $y\in X$ to $v_t(x,y)\in [0,\infty]$ 
that is symmetric in $x$, $y$ and non-increasing in $t$.  
There is one further property worthy of note.  A simple calculation with average
speeds, namely $v_t(x,y)= d(x,y)/t$ where $d(x,y)$ is the distance between $x$ and $y$, shows that 
$$
v_{s+t}(x,y)\le v_s(x,z) + v_t(z,y)
$$
for all $s,t>0$ and all $x,y,z\in X$.

In 2006, Vyacheslav V. Chistyakov \cite{Ch06,Ch} introduced the concept of a \textbf{metric modular} on a set, inspired partly by the classical linear modulars on function spaces employed by Nakano  and other in the 1950s.  Our average speed function is an example of a modular in the sense of Chistyakov.

The concept of a modular on a vector space was introduced by {Nakano} in 1950 \cite{Nakano} and 
refined by Musielak and Orlicz in 1959 \cite{MO}.  In their formulation, a \textbf{modular} 
on a vector space $X$ is a functional $\rho\colon X\to [0,\infty]$ satisfying \begin{enumerate}
	\item 
	$\rho(x)=0$ if and only if $x=0$;
	\item
	$\rho(-x) = \rho(x)$;
	\item
	$\rho(ax+by)\le \rho(x)+\rho(y)$ if $a,b\ge 0$ and $a+b=1$.
	
\end{enumerate} 
A modular $m$ is said to be \textbf{convex} if, 
instead of (3), it satisfies the stronger property
\begin{enumerate}
	\item[(3')] 
	$m(ax+by)\le am(x)+bm(y)$ if $a,b\ge 0$ and $a+b=1$.	
\end{enumerate}
Given a modular $m$ on $X$, the \textbf{modular space}
is defined by
$$
X_m = \{x\in X: m(ax)\to 0 \text{ as } a\to 0\}
$$
It is possible to define a corresponding F-norm (or a norm when $m$ is convex) on the modular space.  The Orlicz spaces $L^\varphi$  are examples of this construction \cite{Rao}.

The notion of a metric modular on an arbitrary set
was introduced in 2006 by Chistyakov as a generalization 
of these ideas.  Chistyakov's formulation provides a good framework in which to study fixed point phenomena.

\section{Modular metric spaces}
We start with the definition given by Chistyakov.
\begin{defi}
	Let $X$ be a nonempty set. A \textbf{metric modular} on $X$
	is a function
	%$\lambda \in (0,\infty)$, functions $w\colon (0,\infty) \times X \times X \to [0,\infty)$ will be written as $w_\lambda(x,y)=w(\lambda,x,y)$ for all $\lambda>0$ and $x,y \in X$.
	$$ w \colon  (0,\infty)\times X \times X\to [0,\infty] \,, $$
	written as  $(\lambda,x,y)\mapsto w_\lambda(x,y)$, that satisfies the following axioms 
	for all $x,y,z\in X$:
	
	\begin{enumerate}
		\item
		$w_\lambda(x,y)=0 $ if and only if   $x=y$;
		
		\smallskip
		\item
		$w_\lambda(x,y)=w_\lambda(y,x)$  for all $\lambda>0$;
		
		\smallskip
		\item
		$w_{\lambda+\mu}(x,y)\le w_\lambda(x,z)+w_\mu(y,z)$  for all ${\lambda,\mu}>0$.
	\end{enumerate}
\end{defi}

If the context is clear, we  refer to a metric modular simply as a \textbf{modular}.
A modular $w$ is said to be \textbf{strict} if 
it has the following property: given $x,y\in X$ with $x\neq  y $, we have $w_\lambda (x,y) > 0 $ for all $\lambda>0$. 
\bigskip 

A modular $ w$ on $ X$ is said to be \textbf{convex} if, instead if (3),  it satisfies the stronger inequality
\begin{equation}\label{conv}
w_{\lambda+\mu}(x,y) \le \frac{\lambda}{\lambda+\mu}w_\lambda(x,z)+\frac{\mu}{\mu+\lambda}w_\mu(z,y),
\end{equation}
for all $ \lambda,\mu >0 $ and $ x,y,z \in X $.

%\textbf{Examples}

\medskip
Let $(X,d)$ be a metric space with at least two points. There are several ways to define a metric modular on $X$. %Let $ \lambda >0$  and $x,y\in X$.

\begin{example}
	$$w_\lambda(x,y)=d(x,y)$$
	In this case, property (3) in the definition of a 
	modular is just the triangle inequality for the metric.
	This modular is not convex - just take $z=y$ and $\mu = \lambda$ in (\ref{conv}).
\end{example}	
\begin{example}
	$$w_\lambda(x,y)=\frac{d(x,y)}{\lambda}$$
	In this case, we can think of $w_\lambda(x,y)$
	as the average velocity required to travel from $x$ 
	to $y$ in time $\lambda$. A simple calculation 
	with the triangle inequality shows that this modular
	is convex. 	
\end{example}
\begin{example}
	$$
	w_\lambda(x,y)=
	\begin{cases}
	\infty\quad\text{if}\quad\lambda<d(x,y),\\
	0\quad\text{if}\quad\lambda\ge d(x,y)
	\end{cases}
	$$
	This simple example could be seen as an extreme case of the velocity metaphor --- 
	if the time available is less than $d(x,y)$, then it is impossible
	to travel from $x$ to $y$, but if the time
	is at least $d(x,y)$ then we can travel instantaneously. This modular is also convex.	
\end{example}

We now look at some of the basic properties of modulars.
\begin{proposition} \label{mod}
	Let $w$ be a modular on the set $X$.
	\begin{enumerate}
		\item[(a)] For every $x,y\in X$, the function $\lambda\mapsto
		w_\lambda(x,y)$ is non-increasing.

		\item[(b)]
		Let $w$ be a convex modular.
		For $x,y\in X$, if $w_\lambda(x,y)$ is finite
		for at least one value of $\lambda$, then
		$
		w_\lambda(x,y) \to 0$  as  $\lambda\to  \infty$
		and 
		$w_\lambda(x,y)\to \infty$  as  $\lambda\to 0+$.
		
		\item[(c)] If $w$ is a convex modular, then
		the function $v_\lambda(x,y)= \smash{\dfrac{w_\lambda(x,y)}{\lambda}}$ is a modular on $X$.
	\end{enumerate}
\end{proposition}

\begin{proof}  \mbox{}\\
	(a)  follows from property (3) of modulars, taking $z=y$.\\
	(b) Taking $z=y$ in equation (\ref{conv}) gives
	$$
	w_{\lambda'}(x,y) \le \frac{\lambda}{\lambda'} w_\lambda(x,y)
	$$
	whenever $\lambda'>\lambda$.  
	Choosing $\lambda$ for which $w_\lambda(x,y) $ is finite, we see that
	$w_{\lambda'}(x,y)\to 0$ as $\lambda'\to\infty$.
	And if we fix $\lambda'$ such that
	$w_{\lambda'}(x,y)<\infty$, we get $w_\lambda(x,y)\to \infty$ as 
	$\lambda \to 0+$.\\
	(c) It is obvious that $v$ satisfies the first two properties of a modular Property (3) for $v$ follows easily from the convexity condition
	on $w_\lambda$.
\end{proof}

\section{Modular sets and modular convergence}
Given a modular $w$ on $X$  and a point $x_0$ in $X$, the two sets 
$$
X_w(x_0)=\{x\in X:w_\lambda(x,x_0)\to 0 \text{ as }
\lambda \to \infty\}
$$
and
$$
X^*_w(x_0)=\{x\in X:\exists\lambda>0 \text{ such that }\ w_\lambda(x,x_0)<\infty\}
$$
are each known as \textbf{Modular Sets} around $x_0$.
These sets can be thought of as comprising all the points that are ``accessible" in some sense from $x_0$. 
In both cases, the modular sets form a partition of $X$.  In our motivating example, the modular sets are the individual  land masses.

It is clear that
$ X_w(x_0) \subset X_w^*(x_0)$ in general and Example 1 shows that this inclusion may be strict.  
While there is some ambiguity in using the same term for these two types of sets, 
in the sequel we shall only be working  with modular sets of the form $X^{*}_{w}(x_0)$. 

\begin{proposition}
	If $w$ is a convex modular on $X$, then 
	$$ X_w(x_0) = X_w^*(x_0)\,.$$
\end{proposition}
\begin{proof}
	This follows immediately from Proposition \ref{mod} (b).
\end{proof} 

We now turn our attention to some notions of convergence.

\begin{defi}
	Let $w$ be a modular on $X$.		
	A sequence  $(x_n)$ in $X$ is said to be  \textbf{$w$-convergent} 
	(or \textbf{modular convergent}) to an element $x \in X$ if there exists 
	a number $ \lambda >0$, possibly depending on $ (x_n)$ and $x$, such 
	that $\lim_{n \to \infty}w_\lambda(x_n,x)=0$
	A sequence $(x_n)$ in $X$ is said to be \textbf{$w$--Cauchy} 
	if there exists $\lambda>0$, possibly
	depending on the sequence, such that $w_\lambda(x_m, x_n) \to 0$ as 
	$m,n \to \infty$. $X$ is said to be \textbf{$w$--complete} 
	if every $w$--Cauchy sequence is $w$--convergent in the more 
	precise sense that if
	$w_\lambda(x_m,x_n)\to 0$ as $m,n\to\infty$ for some 
	$\lambda>0$, then there exists $x\in X$
	such that $w_\lambda(x_n,x)\to 0$ as $n\to \infty$
	(with the same  $\lambda$.)
\end{defi}

\textbf{Metrics on the modular set}

Let $w$ be a modular on $ X$ and let $X_w$  be any one
of the modular sets defined by $w$.  Then the formula
$$
d_w(x,y)=\inf\{\lambda>0: w_\lambda(x,y)\le \lambda\},
\quad \forall x,y \in X_w 
$$
defines a metric on $X_w$ \cite{Ch}.

If the modular $w$ is convex, then the modular space can be endowed with another a metric $ d^*_w $ given by 
$$
d_w^*(x,y)=\inf\{\lambda>0: w_\lambda(x,y)\le 1\}\,. 
$$
These metrics on the modular set are strongly equivalent:
$$ 
d_w \le d^*_w \le 2 d_w \,.
$$ 
We refer to \cite{Ch} for details.
\bigskip

The following result shows the relationship between modular and metric convergence.  The proof can be found in \cite{Ch}.

\begin{proposition}
	
	Let $w$ be a convex modular on $X$. 
	Given a sequence $x_n$ for $X_w^*(x_0)$ and an element 
	$x \in X_w^*(x_0)$, we have:
	$$\lim_{n \to \infty} d_w^*(x_n,x)=0 \iff  \lim_{n \to \infty}w_\lambda(x_n,x)=0 \text{ for every $\lambda>0$.}
	$$
	
\end{proposition}
Chistyakov gives an example to show that modular convergence is strictly weaker
than metric convergence in general \cite{ch1}

\section{Fixed Point Theorems}

In 2011, {Chistyakov} generalized the Banach fixed point theorem to the setting of modular metric spaces.  Consider the definition of a contraction on a metric space:
%introduced the notion of fixed point theorem in modular metric spaces and gave the statement of this theorem. We present the short proof of this theorem employing two formulas for modular metric space that are inspired by \textbf{Palais'} \cite{Richard} Fundamental Contraction Inequality.
$d(Tx,Ty) \le k d(x,y)$ for all $x$, $y$, where $k$ is some constant 
satisfying $0\le k <1$.
Looking at Examples 1 and 2, we see that there are at least two ways to generalise this 
to modular metric spaces.
Chistyakov gives two definitions:

\begin{defi}
	Let $w$ be a modular on a set  $X$ and let $X_w^*$ be a modular set.
	A mapping  $T \colon X_w^* \to X_w^* $  is said to be \textbf{ modular contractive}
	(or a \emph{$w$-contraction}) if there exist numbers $ k \in (0,1)$ and $\lambda_0>0$  such that 
	$$
	w_{k \lambda}(Tx,Ty)\le w_\lambda(x,y)
	$$ 
	for all $0 < \lambda \le \lambda_0$ and all  $x,y \in X_w^*$.
\end{defi}

His second definition is stronger:

\begin{defi}
	A mapping  $ T\colon X^*_w \to X^*_w $ is said to be \textbf{ strongly modular contractive}  (or a \emph{strong $w$-contraction})if there exist numbers $ 0 <k < 1 $  and $ \lambda_0>0 $ such that $$w_{k \lambda}(Tx,Ty) \le k w_\lambda (x,y)$$ for all $ 0 <\lambda \le \lambda_0 $ and  all $x,y \in X^*_w$.
\end{defi}
%This formula can also be written: $$ w_\lambda (Tx,Ty)\le k w_{k^{-1} \lambda }(x,y)$$
%\bigskip

Chistyakov proved fixed point theorems for modular contractive and 
strongly modular contractive mappings. 
%Rather than follow
%his proofs, our approach to these results 
%is inspired by Richard Palais's proof of the Banach fixed point theorem %\cite{Richard}.
We consider whether the approach taken by Palais
in his proof of the Banach fixed point theorem \cite{Richard}
can be 
adapted to prove Chistyakov's results.

Suppose that $T$ is a contraction on a metric space $(X,d)$ with contraction constant $k$.
Thus, we have $d(Tx,Ty)\le k d(x,y)$ for all $x$, $y\in X$.  Combining this with an application
of the triangle inequality to the points $x$, $y$, $Tx$ and $Ty$, we get the inequality
$$
d(x,y) \le \frac{d(x,Tx)+ d(y,Ty)}{1-k}\,.
$$
Palais called this the \textbf{Fundamental Contraction Inequality}.  It is a key ingredient in his proof, where it is used to establish the Cauchy property for the sequence generated by iterating the mapping $T$ on some initial point.

We begin with a variant of Palais's inequality for modular
contractive mappings.

\begin{proposition}[Fundamental Modular Contraction Inequality]
	Let $w$ be a convex  modular in $X$, let $T \colon X^*_w \to X^*_w$  be a modular contractive mapping, with
	$w_{k \lambda}(Tx,Ty)\le w_\lambda(x,y)$ for $0<\lambda \leq \lambda_0$. Let $\lambda_1, \lambda_2\ge0$, $\lambda_1 +\lambda_2=(1-k)\lambda$, where
	$0< \lambda <\lambda_0$. 
	Then 
	\begin{equation}\label{fund1}
	w_\lambda (x,y)\le \frac{\lambda_1 w_{\lambda_1}(x,Tx)+\lambda_2 w_{\lambda_2} (y,Ty)}{\lambda (1-k)}
	\end{equation} for every $x,y \in X_w^*$ such that $w_\lambda(x,y)<\infty$.
\end{proposition}

\begin{proof}
	By the convex property, taking $ \lambda=\lambda_1+k \lambda+\lambda_2$, we get
	$$ w_{\lambda_1+k \lambda +\lambda_2} (x,y) \le \frac{\lambda_1}{\lambda } w_{\lambda_1}(x,Tx)+ k w_{k \lambda}(Tx,Ty)+\frac{\lambda_2}{\lambda} w_{\lambda_2} (y,Ty)$$
	since $ \lambda_1 +\lambda_2 = (1-k)\lambda $ and $ w_{k \lambda} (Tx,Ty) \le w_\lambda (x,y)$.
	
	Therefore, 
	$$ w_\lambda(x,y) \le \frac{\lambda_1}{\lambda} w_{\lambda_1} (x,Tx)+k w_\lambda (x,y)+\frac{\lambda_2}{\lambda} w_{\lambda_2}(y,Ty)\,.$$
	Hence
	$$ w_\lambda (x,y)= \frac{\lambda_1 w_{\lambda_1}(x,Tx)+\lambda_2 w_{\lambda_2} (y,Ty)}{\lambda (1-k)} \,.$$
\end{proof}

We now give  the first  fixed point theorem on modular metric spaces by Chistyakov.
Palais' proof of the Banach fixed point theorem makes use
of the Fundamental Contraction Inequality.
Unfortunately, the finiteness condition
in our Fundamental Modular Contraction Inequality 
prevents us from taking the Palais approach to prove 
Chistyakov's fixed point theorem.

\begin{theorem}[Chistyakov \cite{ch1}]
	Let $ w$ be a strict convex modular on $X$ such that the modular 
	space $X^*_w $ is $w$--complete and let $ T\colon X^*_w \to X^*_w $ be a 
	$w$-contractive map such that for each $\lambda >0 $ there exists an 
	$x=x(\lambda) \in X^*_w $ such that $w_\lambda(x,Tx) < \infty $.
	Then $T$ has a fixed point $x_*$ in $X^*_w $. If the modular $w$ assumes only finite values on $X ^*_w$, then the condition $ w_\lambda(x,Tx)< \infty $ is redundant, the fixed point $x_*$ of $T$ is unique and for each $x_0 \in X^*_w $ the sequence of iterates ${T^n x_0}$ is modular convergent to $x_*$. 		
\end{theorem}

Chistyakov's second fixed point theorem drops the convexity assumption on the modular, replacing it with the requirement that the mapping be strongly contractive.
We have another variation of Palais's inequality for these mappings.

\begin{proposition}[Fundamental Strong Modular
	Contraction Inequality]
	Let $w$ be a modular on $X$ and let $ T \colon X_w^* \to X_w^*$ be strongly modular contractive mapping,
	with
	$w_{k \lambda}(Tx,Ty)\le kw_\lambda(x,y)$ for $0<\lambda \leq \lambda_0$. Let $\lambda_1,\lambda_2 \ge0$, $\lambda_1+\lambda_2=(1-k)\lambda$, $ 0<\lambda<\lambda_0$.  Then
	\begin{equation}\label{fund2} w_{\lambda}(x,y)\le \frac {w_{\lambda_1}(x,Tx)+w_{\lambda_2}(y,Ty)} {1-k} 
	\end{equation}
	for every $x,y \in X_w^*$ such that $w_\lambda(x,y)<\infty$.
\end{proposition}

\begin{proof}	
	By property (3) in the definition of a modular, we get	
	$$w_{\lambda_1+k\lambda +\lambda_2}(x,y) \le w_{\lambda_1}(x,Tx)+w_{k\lambda} (Tx,Ty)+ w_{\lambda_2}(y,Ty) $$
	Using $ \lambda_1+\lambda_2 =(1-k)\lambda $ with the  strong contractive property of $T$, 
	$$w_\lambda(x,y)\le w_{\lambda_1}(x,Tx)+kw_\lambda(x,y)+w_{\lambda_2}(y,Ty)$$ 
	and so 
	$$w_\lambda(x,y)\le \frac{w_{\lambda_1}(x,Tx)+w_{\lambda_2}(y,Ty)}{1-k}$$
	for all $x,y \in X^*_w$.
\end{proof}
The finiteness condition in the Fundamental Strong Contraction Modular Inequality
prevents us from using it to prove Chistyakov's second fixed point theorem.

\begin{theorem}[Chistyakov \cite{ch1}]
	Let $w$ be a strict modular on $X$ such that $X_w$ is $w$-complete, 
	let $T\colon X^*_w \to X^*_w $ be a strongly $w$-contractive map such 
	that for each $\lambda >0$  there exists an $x=x(\lambda)\in X^*_w$ 
	such that $w_\lambda (x,Tx)<\infty$ holds. Then $T$ has a fixed point 
	$x_*$ in $X_w^*$. If, in addition, $w$ is finite valued on $X^*_w$, 
	then the condition $w_\lambda (x,Tx)<\infty $ is redundant, the fixed 
	point $x_*$ of $T$ is unique and for each $x_0 \in X^*_w $ the sequence 
	of iterates ${T^n x_0}$ is modular convergent to $x_*$.	
\end{theorem}

\section{Applications}
Electrorheological fluids are liquids that rapidly solidify in the presence of an electric field.  
They are often studied using Sobolev spaces with a variable exponent \cite{MR}.
It has been suggested that  modular metric spaces may be useful in modelling them \cite{AK}.

Recent work indicates that modular metric space fixed point results are well adapted to certain types of differential equations \cite{Ch}.  Finally, 
we refer to \cite{ChII} for a detailed study of nonlinear superposition operators on modular metric spaces of functions.

\bigskip
\noindent
\textbf{Acknowledgement}\\
This article is taken from the first author's MSc thesis,
which was written under the supervision of the second author.
The first author is grateful to Dr Ryan and to the staff of the School of 
Mathematics at NUI Galway for all the support and encouragement they gave her.
The work of the first author was supported by a grant from the Ministry of Education of Libya, to whom she expresses her gratitude.

%\begin{thebibliography}{11}
	% replace 11 by 1 if fewer than 10 bibitems, 
	% and by 111 if more than 100.
%\end{thebibliography}

\providecommand{\bysame}{\leavevmode\hbox to3em{\hrulefill}\thinspace}
\providecommand{\MR}{\relax\ifhmode\unskip\space\fi MR }
\providecommand{\MRhref}[2]{%
	\href{http://www.ams.org/mathscinet-getitem?mr=#1}{#2}
}
\providecommand{\href}[2]{#2}

\bigskip
\noindent
{\small
\authorbio
}

\end{document}